\documentclass{amsart}
\usepackage{amsfonts, amsmath, amssymb, amsthm}

\newtheorem{theorem}{Theorem}[section]
\newtheorem{remark}[theorem]{ Remark}

\newtheorem{proposition}[theorem]{Proposition}

\newtheorem{lemma}[theorem]{Lemma}

\newcommand{\RR} {\mathbb R}

\newcommand{\HH} {\mathbb H}
\newcommand{\V}{\Vert}

\newcommand{\pa} {\partial}
\newcommand{\Cal} {\mathcal}

\newcommand{\beq} {\begin{equation}}
\newcommand{\eeq} {\end{equation}}

\begin{document}
\title[The Weinstein functional on Riemannian manifolds ]{ Extremal values of the (fractional) Weinstein functional on the hyperbolic space }

\author{Mayukh Mukherjee}
\thanks{The author was partially supported by NSF grant  DMS-1161620.}

\address{Department of Mathematics\\ UNC-Chapel Hill\\ CB \#3250, Phillips Hall
\\ Chapel Hill, NC 27599}

\email{mayukh@live.unc.edu}
\begin{abstract} 
We make a study of Weinstein functionals, first defined in ~\cite{W}, on the hyperbolic space $\HH^n$. We are primarily interested in the existence of Weinstein functional maximisers, or, in other words, existence of extremal functions for the best constant of the Gagliardo-Nirenberg inequality. The main result is that the maximum value of the Weinstein functional on $\HH^n$ is the same as that on $\RR^n$ and the related fact that the maximum value of the Weinstein functional is not attained on $\HH^n$, when maximisation is done in the Sobolev space $H^1(\HH^n)$. This proves a conjecture made in ~\cite{CMMT} and also answers questions raised in several other papers (see, for example, ~\cite{B}). We also prove that a corresponding version of the conjecture will hold for the Weinstein functional with the fractional Laplacian as well.
\end{abstract}
\maketitle

\section{\bf Introduction}
The Weinstein functional on a manifold $M$ for a function $u$ is defined by
\beq
W(u) = \frac{\V u\V^{p + 1}_{L^{p + 1}}}{\V u\V^\alpha_{L^2}\V \nabla u\V^\beta_{L^2}}
\eeq
with $\alpha = 2 - (n - 2)(p - 1)/2, \beta = n(p - 1)/2, n = \text{dim}(M)$. We also keep $p$ in the range $( 1, \frac{n + 2}{n - 2})$ unless otherwise mentioned. We are interested in whether $W(u)$ {\em attains} a maximum over $H^1(M)$. It is clear that if the Gagliardo-Nirenberg inequality
\beq \label{GNI}
\V u\V^{p + 1}_{L^{p + 1}} \leq C\V u\V^\alpha_{L^2}\V \nabla u\V^\beta_{L^2}
\eeq
holds, then $W(u)$ is bounded above, and moreover, the best constant in the Gagliardo-Nirenberg inequality will also be the supremum of the Weinstein functional over $H^1(M)$, denoted by $W_M^{\text{sup}}$. As a notational convenience, we will sometimes drop the subscript $M$ when there is no cause for confusion.\newline
The functional was first introduced in ~\cite{W} to study the bound states for nonlinear Schr{\"o}dinger equations. Now why is it important? Consider the nonlinear Schr{\"o}dinger equation 
\begin{align}\label{moloja}
iv_t + \Delta v & + |v|^{p - 1}v = 0, x \in M, \text{   }v(0, x) = v_0(x).
\end{align}
A nonlinear bound state/standing wave solution of (\ref{moloja}) is a choice of an initial condition $u_\lambda (x)$ such that 
\[
v(t, x) = e^{i\lambda t} u_\lambda(x).
\]
Plugging in this ansatz in (\ref{moloja}) yields the following auxiliary elliptic equation
\beq \label{aux}
-\Delta u_\lambda + \lambda u_\lambda - |u_\lambda|^{p - 1}u_\lambda = 0.
\eeq
We also note that seeking a standing wave solution to the nonlinear Klein-Gordon equation 
\beq\label{molochai}
v_{tt} - \Delta v + m^2 v - |v|^{p - 1}v = 0,\text{   } v(t, x) = e^{i\mu t}u(x)
\eeq
will lead to\footnote{From the point of view of standing waves, there is no essential difference in the analyses of the NLS and the NLKG.} (\ref{aux}) with $\lambda = m^2 - \mu^2$.\newline
Now, with $u, v \in H^1(M)$, we calculate that,
\beq\label{no}
\frac{d}{d\tau}W(u + \tau v)\bigg|_{\tau = 0} = \frac{\text{Re}(N(u), v)}{\V u\V^{2 \alpha}_{L^2}\V \nabla u\V^{2\beta}_{L^2}},
\eeq
where 
\[
N(u) = (p + 1)\V u\V^\alpha_{L^2}\V \nabla u\V^\beta_{L^2} |u|^{p - 1}u - \beta \V u\V^{p + 1}_{L^{p + 1}}\V u\V^\alpha_{L^2}\V \nabla u\V^{\beta - 2}_{L^2}(-\Delta u) - \alpha \V u\V^{\alpha - 2}_{L^2}\V \nabla u\V^\beta_{L^2}u\V u\V^{p + 1}_{L^{p + 1}}.
\]
Let $u$ be a maximiser of the Weinstein functional, and let 
\beq 
\lambda = \frac{\alpha}{\beta}\frac{\V \nabla u\V^2_{L^2}}{\V u\V^2_{L^2}}, \text{  } K = \frac{p + 1}{\beta}\frac{\V \nabla u\V^2_{L^2}}{\V u\V^{p + 1}_{L^{p + 1}}}.
\eeq
Then, (\ref{no}) shows that $u$ will give a solution to 
\beq \label{jalale}
-\Delta v + \lambda v = K|v|^{p - 1}v.
\eeq 
Now, if $u$ solves (\ref{jalale}), then $u_a = au$ solves 
\begin{equation}
-\Delta v + \lambda v = |a|^{1 - p}|v|^{p - 1}v,
\end{equation}
which finally means that we can solve (\ref{aux}) for any $K > 0$. This allows us to 
pass in between (\ref{jalale}) and (\ref{aux}). \newline
Theorem B of ~\cite{W} establishes the existence of a maximiser of the Weinstein functional inside $H^1(\RR^n)$. The main objective of Weinstein's work was to establish a sharp criterion for the existence of global solutions to the focusing nonlinear Schr{\"o}dinger equation on $\RR^{+} \times \RR^n$:
\beq\label{kevinnash}
iv_t + \Delta v = - \frac{1}{2}|v|^{p - 1}v, \text{  } v(x, 0) = v_0 (x),
\eeq
in the energy critical case $p = 1 + 4/n$. Before his work, (\ref{kevinnash}) was already known to have global solutions for any $v_0 \in H^1(\RR^n)$ with $\V v_0\V_{L^2}$ sufficiently small. The question was: exactly how small? This was answered in the energy critical case by
\begin{theorem} (Weinstein ~\cite{W}) Let $v_0 \in H^1(\RR^n)$. For $p = 1 + 4/n$, a sufficient condition for global existence in the initial-value problem (\ref{kevinnash}) is
\[
\V v_0\V_{L^2} < \V \psi \V_{L^2},
\]
where $\psi$ is a positive solution of the equation
\[
-\Delta u + u = u^{1 + 4/n}\]
of minimal $L^2$ norm.
\end{theorem}
Such solutions of minimal $L^2$ norm are also known as ground states.  Theorem B of ~\cite{W} shows that the Weinstein functional maximiser exists in $H^1(\RR^n)$ and also that it gives a ground state solution to (\ref{kevinnash}).\newline
In the setting of the hyperbolic space, consider the focusing nonlinear Schr{\"o}dinger equation
\beq\label{shawnmichaels}
iv_t + \Delta_{\HH^n} v = -|v|^{p - 1}v, \text{  } v(0, x) = v_0 \in H^1(\HH^n).
\eeq
We know that the Gagliardo-Nirenberg inequality holds on $\HH^n$ (see, for example, ~\cite{Ba}, Subsection 6.1). Let $C$ represent the best constant of the Gagliardo-Nirenberg inequality or $W_{\HH^n}^{\sup}$ in the energy critical case $p = 1 + 4/n$. Then, as stated in ~\cite{Ba}, (\ref{shawnmichaels}) has global solution if 
\[
\V v_0\V_{L^2} < (\frac{2 + 4/n}{2C})^{n/4}.
\]
Now, we can raise the following question: how do the best constants of the Gagliardo-Nirenberg inequality on $\RR^n$ and $\HH^n$ compare? It is known that the best constant in the Gagliardo-Nirenberg inequality on $\HH^n$ is greater than or equal to the one on $\RR^n$ (see Proposition \ref{2.2parina} below), but not obviously equal to it. This motivates us to investigate this natural question in Section \ref{hghg} below. In this regard, also refer to Section 4.3 of ~\cite{CMMT}.\newline
Applications to Schr{\"o}dinger equations apart, the Weinstein functional is an interesting nonlinear functional in its own right, and establishing where it can be maximised (that is, there exists a function which attains the maximum) is intrinsically related to the geometry of the manifold $M$ and can be quite tricky. The functional is not at all well-behaved with respect to  conformal changes of the metric, which adds to the difficulty.  To the best of our knowledge, the question of existence of Weinstein functional maximisers is largely unexplored in the compact setting, for example, on compact manifolds with boundary with Dirichlet boundary conditions.\newline
In the setting of non-compact Riemannian manifolds, it is not even clear when the Gagliardo-Nirenberg inequality (\ref{GNI}) holds, let alone existence of Weinstein functional maximisers. For the sake of completeness, we recall that the Gagliardo-Nirenberg inequality is implied by any of the following equivalent statements (we will prove a more general version of this implication later on):
\begin{itemize}
\item the heat kernel $p(t, x, y)$ of the manifold $M$ satisfying 
\beq \label{HKBWF}
p(t, x, y) \leq Ct^{-n/2}, t > 0, x, y \in M,
\eeq 
where $C$ is a constant independent of $x, y \text{  and   } t$.
\item Existence of Sobolev embeddings of the form 
\beq \label{SEWF}
\bigg(\int_M |u|^{2n/(n - 2)} dM\bigg)^{(n - 2)/n} \lesssim \int_M |\nabla u|^2 dM, \mbox{   }\forall \mbox{   } u \in C^\infty_0(M).
\eeq
\end{itemize}
\label{sym9} In fact, the above two statements are equivalent. For details on the proofs, see ~\cite{N} and ~\cite{V}. To be specific, ~\cite{N} establishes the heat kernel bounds starting from the Sobolev embeddings given by (\ref{SEWF}). ~\cite{V} has the converse. \newline
In particular, among other things, it is known that non-negative lower bounds on the Ricci curvature implies any of the above (actually the lower bound on the Ricci curvature is a much stronger condition; it can even imply Gaussian bounds on the heat kernel, see ~\cite{SY}). The heat kernel bounds (\ref{HKBWF}) are known separately for the hyperbolic space and many other nice rank 1 symmetric spaces (also see ~\cite{HS}). In any case, we know that $W_M^{sup}$ exists at least when $M = \RR^n, \HH^n$ as well as on compact manifolds with boundary with Dirichlet boundary conditions. We must also state the obvious at this point: the Weinstein functional maximisation problem does not make sense on a compact manifold without boundary, as the constants would make the $\V \nabla u\V^\beta_{L^2}$ term on the denominator vanish. One of the better ways to make sense of the problem on a compact manifold $M$ with boundary is to use Dirichlet boundary conditions; it disallows one from plugging in nonzero constant functions $u$ into $W(u)$. A Weinstein functional maximiser in $H^1_0(M)$ will give a solution to (\ref{aux}) with Dirichlet boundary conditions. \newline
Before we state our main theorems, let us begin with a few preliminary lemmas. The first thing we want to point out is the following
\begin{lemma}
Scaling the metric has no effect on the Weinstein functional. In other words, consider a manifold $(M, g)$ and the same (smooth) manifold with a scaled metric $(M, rg)$ ($r > 0$). Let $W(u)$ and $W_r(u)$ represent the Weinstein functionals of $u$ with respect to the metrics $g$ and $rg$ respectively. Then
\[
W_r(u) = W(u),
\]
which implies that
\[
W^{\sup}_{(M, g)} = W^{\sup}_{(M, rg)}.
\]
\end{lemma}
\begin{proof}
Let $g_r = rg$ be the scaled metric and $\nabla_r$ denote the gradient 
of $u$ with respect to $g_r$. Then\label{sym22},
\begin{align}\label{900}
\int_M |u|^{p + 1}\sqrt{g_r}dx & = r^{n/2}\int_M |u|^{p + 1}\sqrt{g}dx,
\end{align}
\begin{align}\label{901}
(\int_M |u|^2\sqrt{g_r}dx)^{\alpha/2} & = r^{\alpha n/4}(\int_M |u|^{2}\sqrt{g}dx)^{\alpha/2},
\end{align}
Also, $|\nabla_r u|^2 = \frac{1}{r}|\nabla u|^2$, which means
\begin{align}\label{902}
\V \nabla_r u\V^\beta_{L^2} & = r^{\beta n/4 - \beta/2}\V\nabla u\V^\beta_{L^2}.
\end{align}
Finally, from (\ref{900}), (\ref{901}) and (\ref{902}), we have that $W_r(u) = W(u)$.
\end{proof}
So let us talk about one consequence of this lemma. Consider any manifold $M$ of dimension $n$. Then (also c.f. ~\cite{CMMT}, (4.3.18)), we have 
\begin{proposition}\label{2.2parina}
\beq
W^{\text{sup}}_M \geq W^{\text{sup}}_{\RR^n}.
\eeq
\end{proposition}
\begin{proof}
Start by selecting an open ball $U \subset M$ small enough so that it is diffeomorphic to the Euclidean 1-ball. When we scale the metric $g \mapsto g_r = rg$, as $r \to \infty$, let $U_r$ denote the dilated ball obtained from $U$. We see that $U_r$ approaches $\RR^n$ as $r \to \infty$. Then, using the scaling independence of $W(u)$, we have,
\[
W^{\text{sup}}_{\RR^n} = \lim W^{\text{sup}}_{U_r} = \lim W^{\text{sup}}_U = W^{\text{sup}}_U,
\]
where $W^{\text{sup}}_{U_r}$ is taken over all $u \in H^1_0(U_r)$.
Also, since $U \subset M$, 
\beq\label{960}
W^{\text{sup}}_M \geq W^{\text{sup}}_U.
\eeq
\end{proof}
We will describe in a later section how to construct compact manifolds with boundary $\overline{M}$ with the Dirichlet boundary condition for which we have
\[
W^{\text{sup}}_{\overline{M}} > W^{\text{sup}}_{\RR^n},
\]
which will demonstrate that equality does not always hold in (\ref{960}).

\section{\bf Comparing $W^{\text{sup}}_{\HH^n}$ with $W^{\text{sup}}_{\RR^n}$}\label{hghg}
Since the Gagliardo-Nirenberg inequality holds on $\HH^n$, $W^{\text{sup}}_{\HH^n}$ does exist, and as proven in Proposition \ref{2.2parina}, $W^{\text{sup}}_{\HH^n} \geq W^{\text{sup}}_{\RR^n}$. Now we investigate the question whether $W^{\text{sup}}_{\HH^n}$ is attained, or, in other words, whether there exists a Weinstein functional maximiser in $H^1(\HH^n)$. To attack this question, it seems convenient to use the following model of $\HH^n$:
\[
\HH^n = \{v = (v_0, v') \in \RR^{n + 1} : \langle v, v\rangle = 1, v_0 > 0\},
\]
and the metric on $\HH^n$ is given by the restriction of the Lorentzian metric on $\RR^{n + 1}$,
\[g = - d^2_{x_1} + d^2_{x_2} + ... + d^2_{x_{n + 1}}\]
to $\HH^n$. 
Let us parametrize $\HH^n$ using the following ``polar'' model:
\beq\label{polarmodel}
\HH^n = \{(t, x) \in \RR^{1 + n} : t = \text{cosh }r, x = \text{sinh }r\omega, r \geq 0, \omega \in S^{n - 1}\}.
\eeq
We note that the ``polar metric'' of $\HH^n$ is given by 
\beq\label{PMHS}
ds^2 = dr^2 + \text{sinh}^2 rd\omega^2,
\eeq
as compared to the corresponding ``polar'' metric on $\RR^n$, given by
\beq\label{PMES}
ds^2 = dr^2 + r^2d\omega^2.
\eeq
Comparing these two, we define the following map $T : L^2(\RR^n) \longrightarrow L^2(\HH^n)$ by 
\beq\label{94sesh}
T(u) = \phi u,
\eeq
where 
\beq\label{95sesh}
\phi(r) = (\frac{r}{\text{sinh }r})^{\frac{n - 1}{2}}.
\eeq
It is clear that $T$ is an isometry, since
\begin{align}
\int_{\HH^n} |\phi u|^2 d\HH^n & = \int_{r = 0}^\infty\int_{S^{n - 1}} |u|^2 (\frac{r}{\text{sinh }r})^{n - 1}\text{sinh}^{n - 1}rdrd\omega\nonumber\\
& = \int_{r = 0}^\infty\int_{S^{n - 1}} |u|^2r^{n - 1}drd\omega = \int_{\RR^n} |u|^2 d\RR^n.
\end{align}
Now we can state our first main theorem of this paper:
\begin{theorem}{\bf (Main Theorem I)}\label{WFMT}
\beq
W^{\text{sup}}_{\HH^n} = W^{\text{sup}}_{\RR^n}.
\eeq
\end{theorem}
\begin{proof} The following is the scheme of our proof: we show that, given a function $v \in H^1(\HH^n)$, we can find a corresponding function $u \in H^1(\RR^n)$ such that 
\[
W_{\HH^n}(v) < W_{\RR^n}(u).
\]
So, if we can use a map that preserves the $L^2$ norm (we have the map $T$ as defined above in mind), that is, $\V u\V_{L^2(\HH^n)} = \V v\V_{L^2(\RR^n)}$, the major issue to address is how to compare their $L^{p + 1}$ and gradient-$L^2$ norms. That is, we are done if we can show that, with $\phi$ as in (\ref{94sesh}) and (\ref{95sesh}), 
\begin{itemize}
\item $\V \nabla(\phi u)\V_{L^2(\HH^n)} > \V \nabla u\V_{L^2(\RR^n)} $ and 
\item $\V \phi u\V_{L^{p + 1}(\HH^n)} < \V u\V_{L^{p + 1}(\RR^n)}$.
\end{itemize}
To that end, we quote the following calculation from ~\cite{CM}:
\[
\pa_r(\phi) = \frac{n - 1}{2} \bigg(\frac{r}{\text{sinh }r}\bigg)^{\frac{n - 3}{2}}\bigg(\frac{\text{sinh }r - r\text{cosh }r}{\text{sinh}^2(r)}\bigg)
\]
and 
\begin{align*}
\pa^2_r(\phi) & = \bigg(\frac{n - 1}{2}\bigg)\bigg(\frac{n - 3}{2}\bigg)\bigg(\frac{r}{\text{sinh }r}\bigg)^{\frac{n - 5}{2}}\bigg(\frac{\text{sinh }r - r\text{cosh }r}{\text{sinh}^2(r)}\bigg)^2 \\
& \text{   } + \frac{n - 1}{2} \bigg(\frac{r}{\text{sinh }r}\bigg)^{\frac{n - 3}{2}} \bigg(\frac{2r\text{sinh }r\text{cosh}^2(r) - 2\text{sinh}^2(r)\text{cosh }r - r\text{sinh}^3(r)}{\text{sinh}^4(r)}\bigg).
\end{align*}
Then we have
\beq\label{oboseshe}
\begin{aligned}
\phi^{-1}(-\Delta_{\HH^n})(\phi u) & = \phi^{-1}(-\pa^2_r - (n - 1)\frac{\text{cosh }r}{\text{sinh }r}\pa_r - \frac{1}{\text{sinh}^2(r)}\Delta_{S^{n - 1}})(\phi u)\\
& = -\pa^2_ru - 2\phi^{-1}(\pa_r\phi)(\pa_ru) - \phi^{-1}u\pa^2_r\phi \\
& \text{    }- (n - 1)\frac{\text{cosh }r}{\text{sinh }r}\pa_r u - (n - 1)\frac{\text{cosh }r}{\text{sinh }r}\phi^{-1}u\pa_r\phi - \frac{1}{\text{sinh}^2(r)}\Delta_{S^{n - 1}} u\\
& = -\pa^2_r u + V_0(r)\pa_r u + \bigg[V_n(r) + \bigg(\frac{n - 1}{2}\bigg)^2\bigg]u - \frac{1}{\text{sinh}^2(r)}\Delta_{S^{n - 1}} u\\
& = -\Delta' u + \bigg[V_n(r) + \bigg(\frac{n - 1}{2}\bigg)^2\bigg]u,
\end{aligned}
\eeq
where 
\beq\label{perechi}
\begin{aligned}
V_0(r) & = \frac{1 - n}{r}\\
V_n(r) & = \bigg(\frac{n - 1}{2}\bigg)\bigg(\frac{n - 3}{2}\bigg)\frac{1}{\text{sinh}^2 r} - \bigg(\frac{n -1}{2}\bigg)\bigg(\frac{n - 3}{2}\bigg)\frac{1}{r^2}\\& = \bigg(\frac{n -1}{2}\bigg)\bigg(\frac{n - 3}{2}\bigg)V(r)\\
-\Delta' & = -\Delta_{\RR^n} + \frac{\text{sinh}^2 r - r^2}{r^2\text{sinh}^2 r}\Delta _{S^{n - 1}},
\end{aligned}
\eeq
where $V(r) = \frac{1}{\text{sinh}^2 r} - \frac{1}{r^2}$.\newline
Now, start by selecting a radial function $u \in H^1(\RR^n)$. By the preceding calculation, using the fact that $\phi$ is an isometry and $-\Delta_{S^{n - 1}}u = 0$ (since $u$ is radial), we have
\beq\label{sa}
(-\Delta_{\HH^n} \phi u, \phi u) = (-\Delta_{\RR^n} u, u) + \epsilon \V u\V^2_{L^2}
\eeq
for some $\epsilon > 0$, because we have for all $r$ (see justification below), 
\beq \label{eq1}
(n - 3)(\frac{1}{r^2} - \frac{1}{\text{sinh}^2r}) < n - 1
\eeq
when $n \neq 2$ and 
\beq\label{eq2}
(n - 1)(n - 3)(\frac{1}{r^2} - \frac{1}{\text{sinh}^2r}) < 0 < (n - 1)^2
\eeq
when $n = 2$.\newline
Together (\ref{eq1}) and (\ref{eq2}) give us that for all $r > 0$, 
\[V_n (r) + \bigg(\frac{n - 1}{2}\bigg)^2 > 0,
\]
which in turn implies that $\epsilon > 0$. \newline
Let us justify (\ref{eq1}): this can be seen by observing that
\[
\lim_{r \to 0+}V(r) = -1/3
\]
and the fact that $V_n(r)$ does not attain an extremum for any $r > 0$. In fact $V'_n(r) = 0$ only when $r = 0$. This is because, we see that 
\begin{align*}
V'(r) = 0 & \Longrightarrow \frac{\text{sinh}^3r - r^3\text{cosh }r}{r^3\text{sinh}^3r} = 0\\
& \Longrightarrow \frac{\text{sinh}^3r}{\text{cosh }r} = r^3.
\end{align*}
If we let 
\[
h(r) = \frac{\text{sinh }r}{\text{cosh}^{1/3}r},
\]
then proving that $h'(r) > 1$ for all $r > 0$ will suffice, as then $h(r)$ can never equal $r$. Now,
\begin{align*}
h'(r) & = \frac{3\text{cosh}^2r - \text{sinh}^2r}{3\text{cosh}^{4/3}r} = \frac{2\text{cosh}^2r + 1}{3\text{cosh}^{4/3}r}.
\end{align*}
Now, writing $\text{cosh}^2r = z$, we have that
\begin{align*}
\frac{2\text{cosh}^2r + 1}{3\text{cosh}^{4/3}r} \leq 1 & \Longrightarrow 8z^3  - 15z^2 + 6z + 1 \leq 0\\
& \Longrightarrow (z -1)^2(8z + 1) \leq 0 \Longrightarrow z = 1,
\end{align*}
which can only happen if $r = 0$. So everywhere else, we have $h'(r) > 1$.\newline
When $r \to \infty$, $V(r) \to 0-$. Also, the fact that $V(r)$ does not attain an extremum means that $V(r) > -1/3 > -\frac{n - 1}{n - 3}$ always.  \newline
So, finally, from (\ref{sa}) we have that 
\beq \label{g}
\V\nabla (\phi u)\V^2_{L^2(\HH^n)} > \V \nabla u\V^2_{L^2(\RR^n)}.
\eeq
Also, 
when $p > 1$, we have
\begin{align*}
\int_{\HH^n} |\phi u|^{p + 1}d\HH^n & = \int_0^\infty\int_{S^{n - 1}} |u|^{p + 1}\frac{r^{(n - 1)(p + 1)/2}}{\text{sinh}^{(n - 1)(p + 1)/2} r}\text{sinh}^{n - 1}(r) dr d\omega\\
& = \int_0^\infty\int_{S^{n - 1}}|u|^{p + 1}\frac{r^{(n - 1)(p - 1)/2}}{\text{sinh}^{(n - 1)(p - 1)/2}r}r^{n - 1}drd\omega\\
& < \int_0^\infty\int_{S^{n - 1}} |u|^{p + 1}r^{n - 1}dr d\omega = \int_{\RR^n} |u|^{p + 1}d\RR^n.
\end{align*}
So, ultimately, we have, 
\beq\label{103}
W_{\HH^n}(\phi u) < W_{\RR^n}(u).
\eeq
However, it is known that 
\beq\label{rad1}
W^{\text{sup}}_{\HH^n} = \sup \{W_{\HH^n}(u) : u \text{    is a radial function   } \in H^1(\HH^n)\}.
\eeq
For details on this, see ~\cite{CM}\footnote{Also, by a radial function in this context, we mean a function whose value at a point depends solely on the distance of the point from a pre-chosen fixed point, which can be called the origin.}. Heuristically, the basic argument is that we start with an arbitrary function $u$ and then consider its symmetric decreasing rearrangement $u^*$, and make use of the fact that symmetric decreasing  rearrangements keep the same $L^s$-norms for all $s$, that is,
\[
\V u^*\V_{L^s(\HH^n)} = \V u\V_{L^s(\HH^n)}, \text{   }s \in [1, \infty],
\] but they decrease gradient norms, that is,
\beq\label{Seethis}
\V \nabla u^*\V_{L^s(\HH^n)} \leq \V \nabla u\V_{L^s(\HH^n)},\text{  } s \in [1, \infty].
\eeq
 To prove (\ref{Seethis}), ~\cite{CM} writes
\[
\V\nabla f\V_{L^2(\HH^n)} = \lim_{t \to 0}I^t(f),
\]
where 
\[
I^t(f) = t^{-1}[(f, f)_{\HH^n} - (f, e^{t\Delta_{\HH^n}}f)_{\HH^n}],
\]
$(.,.)_{\HH^n}$ denoting the usual inner product in $L^2(\HH^n)$.\newline
Since the symmetric decreasing rearrangement keeps same $L^2$-norm, now one just needs to see 
\beq\label{jnts}
(f^*, e^{t\Delta_{\HH^n}}f^*)_{\HH^n} \geq (f, e^{t\Delta_{\HH^n}}f)_{\HH^n}.
\eeq
Lemma 3.3 of ~\cite{CM} proves (\ref{jnts}) with the help of a rearrangement inequality from ~\cite{D} (which we reproduce below). \newline
Also, the statement for $\RR^n$ corresponding to (\ref{Seethis}) is given by:
\beq\label{Seethistoo}
\V\nabla u^*\V_{L^s(\RR^n)} \leq \V \nabla u\V_{L^s(\RR^n)}.
\eeq
For a proof of (\ref{Seethistoo}), see ~\cite{LL}.\newline
Finally, from what has gone,  
\[
W_{\HH^n}(u^*) \geq W_{\HH^n}(u) \text{  } \forall \text{  } u \in H^1(\HH^n),
\]
which establishes (\ref{rad1}). \newline
Lastly, we mention the fact that it does not matter where the radial functions are centered in the respective spaces, that is, if $\varphi$ is a radial function in $H^1(M)$ ($M = \HH^n$ or $\RR^n$), centered at $P \in M$, and $\psi$ is a translate of $\varphi$ centered at another point $Q \in M$, then $W_M(\varphi) = W_M(\psi)$. Towards that end, let $(1, \overline{0}) \in \HH^n$ be the point $t = 1, x = \overline{0} = (0,.., 0)$, as per the notation of (\ref{polarmodel}). Using the homogeneity of $\HH^n$, we can infer that 
\begin{align*}
\text{sup}\{W_{\HH^n}(u) : u \text{    is a radial function}\} &
 = \text{sup}\{W_{\HH^n}(u) : u \text{    is a radial function  centered at   } (1, \bar{0})\}.
\end{align*}
Also, using (\ref{Seethistoo}) and the homogeneity of $\RR^n$, we have,
\begin{align*}
W^{\text{sup}}_{\RR^n} & = \text{sup}\{W_{\RR^n}(u) : u \text{    is a radial function}\} \\
& = \text{sup}\{W_{\RR^n}(u) : u \text{    is a radial function  centered at } 0 \}.
\end{align*}
So, using (\ref{103}), and the conclusion of Proposition \ref{2.2parina}, we ultimately have our result.
\end{proof}
We include here the aforementioned rearrangement result from ~\cite{D}, as quoted in ~\cite{CM}.
\begin{theorem} \label{Drag}(\text{Draghici ~\cite{D}}) Let $X = \HH^n$, $f_i = X \to R_{+}$\label{sym555} be $m$ nonnegative functions, $\Psi \in AL_2(R^m_+)$ be continuous and $K_{ij} : [0, \infty) \to [0, \infty)$, $i < j$, $j \in \{1,...., m\}$ be decreasing functions. We define 
\[
I[f_1,...,f_m] = \int_{X^m} \Psi(f_1(\Omega_1),..., f_m(\Omega_m))\Pi_{i < j}K_{ij}(d(\Omega_i, \Omega_j))d\Omega_1...d\Omega_m.
\]
Then the following inequality holds:
\[
I[f_1,....,f_m] \leq I[f^*_1,...,f^*_m].\]

\end{theorem}
Theorem \ref{WFMT} was conjectured in ~\cite{CMMT} (also see ~\cite{Ba}). Harris (~\cite{Ha}) had collected some numerical evidence of this phenomenon in the special case $p = n = 2$.\newline
Note that we have also proved another related conjecture in ~\cite{CMMT}, which says in effect that for all $u \in H^1(\HH^n)$, $W(u) < W^{\text{sup}}_{\HH^ n}$, which means that there is no Weinstein functional maximiser in $H^1(\HH^n)$. Let us justify this: in case there exists $v \in H^1(\HH^n)$ such that $W_{\HH^n}(v) = W^{\text{sup}}_{\HH^n}$, then the spherical decreasing rearrangement $v^* \in H^1(\HH^n)$ of $|v|$ also satisfies $W_{\HH^n}(v^*) = W^{\text{sup}}_{\HH^n}$. But then, $u^* = \phi^{-1}v^* \in H^1(\RR^n)$ will satisfy $W_{\RR^n}(u^*) > W_{\HH^n}(v^*)$. By Theorem \ref{WFMT}, this is a contradiction.\newline
Let us comment on the implications of this. Suppose we are trying to maximise $W(u)$ on $H^1(\RR^n)$. We can use the fact that $W(u)$ is invariant under the transformation $u \mapsto au$ and spatial scaling $u(x) \mapsto u(bx)$. Now, given a sequence $u_\nu$ such that $W(u_\nu) \longrightarrow W^{max}_{\RR^n}$ the above mentioned facts allow us to normalise 
\[
\V u_\nu\V_{L^2} = 1, \V \nabla u_\nu\V_{L^2} = 1
\]
and pass to a subsequence, still called $u_\nu$ to find $u$ such that $u_\nu \longrightarrow u$ weak$^*$ in $H^1(\RR^n)$ and prove that $u$ actually maximises $W(u)$ in $H^1(\RR^n)$.\newline
When we try to repeat this argument on $\HH^n$, we can only achieve the normalisation (because we don't  have spatial scaling any more)
\[
\V\nabla u_\nu\V_{L^2} = 1
\]
which implies that $\V u_\nu\V_{L^2}$ and $\V u\V_{L^p}$ are bounded. Now, take a subsequence, still called $u_\nu$ such that $u_\nu \longrightarrow u$ weak$^*$ in $H^1(\HH^n)$. Now, the implication is, if after passing to a further subsequence, we have
$
\V u_\nu\V_{L^2} \longrightarrow A
$, then we must have
\beq
A = 0
\eeq
Assume that $A \neq 0$, if possible. In that case, we have
\beq
\V u_\nu\V^{p + 1}_{L^{p + 1}} \longrightarrow A^\alpha W^{max}_{\HH^n}
\eeq
and $u_\nu \longrightarrow u$ in $L^{p + 1}$-norm, so $\V u\V^{p + 1}_{L^{p + 1}} = A^\alpha W^{max}_{\HH^n}$. But, $\V u\V_{L^2} \leq A$ and $\V\nabla u\V_{L^2} \leq 1$, so
\beq
W(u) \geq \frac{A^\alpha W^{max}_{\HH^n}}{A^\alpha} = W^{max}_{\HH^n}
\eeq
which would give us a minimiser of the Weinstein functional on $H^1(\HH^n)$, which, as we proved, is not possible. 
\begin{remark}
For a generic manifold $M$, we do not have $W^{\sup}_M = W^{\sup}_{\RR^n}$. In fact, consider the following counterexample: \newline
Let $M_k$ be the sphere $S^n$ with a tiny open ball (homeomorphic to \label{sym20} $B_1(0) \subset \RR^n$) of radius $r_k$ removed. As we make the radius of the removed ball $r_k \to 0$, we see that the first eigenvalue $\lambda^{(k)}_1$ of the Laplacian $-\Delta_k$ of $M_k$ goes to $0$, because $M_k$ approaches the sphere $S^n$, whose first eigenvalue is 0. Now, consider a sequence of functions $u^k_l$ such that when $k$ is fixed, $W_{M_k}(u^k_l) \to W^{\text{sup}}_{M_k}$.  Since all the $M_k$'s are compact with uniformly bounded volume, we can find a constant $C$ (independent of $k$) such that $\V u^k_l\V_{L^2} \leq C\V u^k_l\V_{L^{p + 1}}$. Now,
\begin{align*}
\frac{\V u^k_l\V^{p + 1}_{L^{p + 1}(M_k)}}{\V u^k_l\V^\alpha_{L^2(M_k)}\V \nabla u^k_l\V^\beta_{L^2(M_k)}} = \frac{\V u^k_l\V^{p + 1}_{L^{p + 1}(M_k)}\V u^k_l\V^\beta_{L^2(M_k)}}{\V u^k_l\V^{p + 1}_{L^2(M_k)}\V \nabla u^k_l\V^\beta_{L^2(M_k)}} & \geq \frac{\V u^k_l\V^\beta_{L^2(M_k)}}{C^{p + 1}\V \nabla u^k_l\V^\beta_{L^2(M_k)}}, \text{   } .
\end{align*}
So, 
\[
\sup \frac{\V u^k_l\V^{p + 1}_{L^{p + 1}(M_k)}}{\V u^k_l\V^\alpha_{L^2(M_k)}\V \nabla u^k_l\V^\beta_{L^2(M_k)}} \geq \frac{1}{C^{p + 1} {(\lambda^{(k)}_1)}^\beta}. \]
This means that we have $W^{\text{sup}}_{M_k} \to \infty$. \newline
On a compact domain inside $\RR^n$ with Dirichlet boundary condition, it is known via a Harnack inequality argument (see Proposition 4.3.1 of ~\cite{CMMT}) that there is no optimal constant for the Gagliardo-Nirenberg inequality. It is however, an interesting (and largely unanswered) question as to what happens in the case of generic compact manifolds with boundary (with Dirichlet boundary condition). 
\end{remark}
\section{\bf Weinstein functional and fractional Laplacian}
We know that Spec$(-\Delta_{\HH^n}) \subset [\frac{(n - 1)^2}{4}, \infty)$. So the spectral theorem can be applied to define the fractional Laplacian $(-\Delta)^\alpha$, $\alpha \in (0, 1)$. Now we investigate the corresponding Weinstein functional maximisation problem for the fractional Laplacian $(-\Delta)^\alpha$. In other words, we try to investigate what we can say about the maximisation problem for 
\[
W_\alpha(u) = \frac{\V u\V^{p + 1}_{L^{p + 1}}}{\V u\V^\gamma_{L^2}\V (-\Delta)^{\frac{\alpha}{2}}u\V^\rho_{L^2}},
\]
where $\gamma = 2 - (n - 2\alpha)(p - 1)/(2\alpha), \rho = n(p - 1)/(2\alpha)$. We will want $p \in (1, \frac{n + 2\alpha}{n - 2\alpha})$.
The reason for our interest in this is the following: if we consider the fractional NLS of the form 
\begin{align*}
iv_t - (-\Delta)^\alpha v & + |v|^{p - 1}v = 0, x \in M\\
v(0, x) & = v_0(x),
\end{align*}
and plug in  
\[
v(t, x) = e^{i\lambda t} u_\lambda(x),
\]
we get the the following auxiliary elliptic equation
\beq \label{auxf}
(-\Delta)^\alpha u_\lambda + \lambda u_\lambda - |u_\lambda|^{p - 1}u_\lambda = 0.
\eeq
By a similar calculation as before, a maximiser $u$ for the fractional Weinstein functional will solve 
\beq
(-\Delta)^\alpha v + \lambda v = K|v|^{p - 1}v,
\eeq
where
\beq
\lambda = \frac{\gamma}{\rho}\frac{\V (-\Delta)^{\alpha/2}u\V^2_{L^2}}{\V u\V^2_{L^2}}, \text{  } K = \frac{p + 1}{\rho}\frac{\V (-\Delta)^{\alpha/2}u\V^2_{L^2}}{\V u\V^{p + 1}_{L^{p + 1}}}.
\eeq
Also, let us mention here that there has been some recent interest in nonlocal equations of the type (\ref{auxf}). For example, see ~\cite{FL} and references therein.\newline 
Now, the fractional Gagliardo-Nirenberg inequality (the fact that it actually holds is the content of Proposition \ref{2.4} below) implies that $W_\alpha(u)$ is actually bounded from above on both $\RR^n$ and $\HH^n$, when $u$ is chosen from the natural domain of $(-\Delta)^{\alpha/2}$, which is
\[
\mathcal{D}((-\Delta)^{\frac{\alpha}{2}}) = H^\alpha (M) \subset L^q(M), \text{   }\forall q \text{   }\in \bigg[2, \frac{2n}{n - 2\alpha}\bigg], M = \RR^n, \HH^n.
\]
Let us discuss when the fractional Gagliardo-Nirenberg inequality holds. We want to justify (our tacit claim above) that it holds on the hyperbolic space $\HH^n$ and the Euclidean space $\RR^n$. Actually we have, more generally:
\begin{proposition}\label{2.4}
Let $M$ be a complete Riemannian manifold on which the heat kernel satisfies the following pointwise bounds:
\beq\label{ptwise}
|p(t, x, y)| \leq Ct^{-n/2}, \text{   }t > 0,\text{    } x, y \in M,
\eeq
where $C$ is constant independent of $t, x$ and $y$. Then the fractional Gagliardo-Nirenberg inequality 
\[
\V u\V^{p + 1}_{L^{p + 1}} \leq C\V (-\Delta)^{\alpha/2}u\V_{L^2}^\rho\V u\V^\gamma_{L^2}
\]
holds on M, where $\gamma = 2 - (n - 2\alpha)(p - 1)/(2\alpha)$, and $ \rho = n(p - 1)/(2\alpha)$.
\end{proposition}
\begin{proof}
We have,
\begin{align*}
\int_M |u|^{p + 1}dM &  = \int_M|u|^{(p + 1)\theta}|u|^{(p + 1)(1 - \theta)}dM\\
& \leq \V |u|^{(p + 1)\theta}\V_{L^{r'}} \V |u|^{(p + 1)(1 - \theta)}\V_{L^{s'}}\\
& = \V u\V^{(p + 1)\theta}_{L^{r'(p + 1)\theta}} \V u\V^{(p + 1)(1 -\theta)}_{L^{s'(p + 1)(1 - \theta)}}
\end{align*}
where $\frac{1}{r'} + \frac{1}{s'} = 1$.\newline
That means, 
\[
\V u\V_{L^{p + 1}} \leq \V u\V^\theta_{L^{r'(p + 1)\theta}}\V u\V^{1 - \theta}_{L^{s'(p + 1)(1 - \theta)}}.
\]
Let $r'(p + 1)\theta = r$ and $s'(p + 1)(1 - \theta) = s$. So
\[
\V u\V_{L^{p + 1}} \leq \V u\V^\theta_{L^{r}}\V u\V^{1 - \theta}_{L^{s}},
\]
where 
\[
\frac{\theta}{r} + \frac{1 - \theta}{s} = \frac{1}{p + 1}.
\]
Now, 
we can assert that the Hardy-Littlewood-Sobolev estimates 
\[
\V u\V_{L^r} \lesssim \V (-\Delta)^{\alpha/2} u\V_{L^m}
\]
where $r = \frac{nm}{n - \alpha m}$, $0 < \alpha < 1$, $1 < m < \frac{n}{\alpha}$, will follow from the heat kernel bounds (see ~\cite{VSC}, Chapter II, Theorem II.2.4 and the following discussion; also see ~\cite{Bau}). Given that, we now have
\[
\V u\V^{p + 1}_{L^{p + 1}} \lesssim \V (-\Delta)^{\alpha/2} u\V^{\theta(p + 1)}_{L^m}\V u\V^{(1 - \theta)(p + 1)}_{L^s}
\]
with 
\[
\theta(\frac{1}{m} - \frac{\alpha}{n}) + \frac{1 - \theta}{s} = \frac{1}{p + 1}.
\]
In the special case of $m = s = 2$, we retrieve the Gagliardo-Nirenberg inequality in the form that we use here.
\end{proof}
\begin{remark} By ~\cite{DGM}, it is known that the heat kernel bounds (\ref{ptwise}) hold on complete simply connected manifolds of dimension $n$ and sectional curvature less than or equal to 0. This is also true on compact manifolds with the Dirichlet Laplacian. As regards symmetric spaces, a similar heat kernel bound holds on spaces of the form $G_{\mathbb{C}}/G$, where $G$ is a compact Lie group and $G_{\mathbb{C}}$ is the complexification of $G$ (for details, see ~\cite{Ga}).
\end{remark}
 Now we have the second main theorem of this paper:
\begin{theorem}(\text{\bf Main Theorem II})\label{WFMT2}
\[
W_{\alpha, \RR^n}^{sup} = W_{\alpha, \HH^n}^{sup}.
\]
\end{theorem}
\begin{proof}
Morally, as in the proof of Theorem \ref{WFMT}, we want to compare $W_{\alpha, \RR^n}(u)$ with $W_{\alpha, \HH^n}(v)$ for functions $u \in H^\alpha(\RR^n), v \in H^\alpha(\HH^n)$. As usual, we use the isometric isomorphism $T$ defined before that keeps $L^2$-norms same and lowers the $L^{p + 1}$-norm on the hyperbolic side, that is, if $v = Tu$, then
\beq
\V u\V_{L^2(\RR^n)} = \V v\V_{L^2(\HH^n)}, \V u\V_{L^{p +1}(\RR^n)} > \V v\V_{L^{p + 1}(\HH^n)}. 
\eeq
Seeing what has gone before, comparing the supremum values of the fractional Weinstein functionals just amounts to comparing $\V (-\Delta_{\RR^n})^{\frac{\alpha}{2}}u\V_{L^2(\RR^n)}$ with $\V (-\Delta_{\HH^n})^{\frac{\alpha}{2}}\phi u\V_{L^2(\HH^n)}$. Now we use the following functional calculus (see ~\cite{B}; also see Proposition 3.1.12 of ~\cite{H}) 
\[
A^{\alpha} u = \frac{\text{sin} \alpha \pi}{\pi}\int^\infty_0 t^{\alpha - 1}(t + A)^{-1}A udt, \text{   }\forall u \in \mathcal{D}(A),
\]
where $A$ is a sectorial operator on a Banach space $X$ and $0 < \alpha < 1$. Now, it is known that on a Hilbert space $H$, a non-negative self-adjoint operator $A : \mathcal{D}(A) \subset H \longrightarrow H$ is sectorial with $\omega = 0$ (see Chapter 2, Subsection 2.1.1 of ~\cite{H}). \newline
So then, writing $(.,.)_{M}$ for the inner product in $L^2(M)$, where $M = \RR^n, \HH^n$, we get,
\[
((-\Delta_{\HH^n})^{\alpha}\phi u, \phi u)_{\HH^n} = \frac{\text{sin} \alpha \pi}{\pi}\int^\infty_0 \int_{\HH^n}t^{\alpha - 1}(t -\Delta_{\HH^n})^{-1}(-\Delta_{\HH^n}) (\phi u) \overline{\phi u} d\HH^ndt
\]
and
\[
((-\Delta_{\RR^n})^{\alpha} u, u)_{\RR^n} = \frac{\text{sin} \alpha \pi}{\pi}\int^\infty_0 \int_{\RR^n}t^{\alpha - 1}(t -\Delta_{\RR^n})^{-1}(-\Delta_{\RR^n}) u \overline{u} d\RR^ndt.
\]
So we have reduced the problem to comparing 
\[
\int_{\HH^n}(t -\Delta_{\HH^n})^{-1}(-\Delta_{\HH^n}) (\phi u) \overline{\phi u} d\HH^n
\]
with 
\[
\int_{\RR^n}(t -\Delta_{\RR^n})^{-1}(-\Delta_{\RR^n}) u \overline{u} d\RR^n.
\]
Now, if we let $u = u_1 + iu_2$, we will see that for the above comparison it is enough to consider real-valued $u$. More generally, consider a linear self-adjoint operator $L$ on $L^2(M)$ and a function $v = v_1 + iv_2$. Then,
\begin{align*}
\int_M Lv \overline{v} dM & = \int_M L(v_1 + iv_2)(v_1 - iv_2)dM = \int_M (Lv_1 + iLv_2)(v_1 - iv_2)dM \\
& = \int_M Lv_1 v_1 dM + \int_M Lv_2 v_2 dM + i\int_M (-Lv_1 v_2 + v_1Lv_2) dM.
\end{align*}
Here, $L = (\lambda - \Delta_M)^{-1}(-\Delta_M) $, where $M = \RR^n$ or $\HH^n$. If we can prove that for any real valued $\varphi \in \Cal{D}(L)$, $L\varphi$ is real-valued, then the symmetry of $L$ will imply that $\int_M (-Lv_1 v_2 + v_1Lv_2) dM = 0$. \newline
If $\varphi$ is real-valued, then so is $-\Delta \varphi = \psi$. For real-valued $f, g \in L^2(M)$, if $(\lambda - \Delta_M)(f + ig) = \psi$, then that would imply $-\Delta_M g = -\lambda g$ for $\lambda \geq 0$, which is impossible for $\lambda > 0$. Also, $\text{Spec }(-\Delta_{\HH^n}) \subset [\frac{(n - 1)^2}{4}, \infty)$, and since there are no $L^2$ harmonic functions on $\RR^n$, we can rule out $\lambda = 0$. So we have reduced the problem to the comparison of 
\[
A = \int_{\HH^n}(t -\Delta_{\HH^n})^{-1}(-\Delta_{\HH^n}) (\phi u) (\phi u) d\HH^n
\]
with 
\[
B = \int_{\RR^n}(t -\Delta_{\RR^n})^{-1}(-\Delta_{\RR^n}) (u) (u) d\RR^n,
\]
where $u$ is real-valued. So, let us call
\begin{align*}
F(t) & = ((t - \Delta_{\HH^n})^{-1}(-\Delta_{\HH^n})\phi u, \phi u)_{\HH^n} - ((t - \Delta_{\RR^n})^{-1}(-\Delta_{\RR^n}) u, u)_{\RR^n}\\
& = ((t - \phi^{-1}\Delta_{\HH^n}\phi)^{-1}(-\phi^{-1}\Delta_{\HH^n}\phi) u, u)_{\RR^n} - ((t - \Delta_{\RR^n})^{-1}(-\Delta_{\RR^n}) u, u)_{\RR^n}\\
& = (((t - \overline{\Delta})^{-1}(-\overline{\Delta}) - (t - \Delta_{\RR^n})^{-1}(-\Delta_{\RR^n})) u, u)_{\RR^n},
\end{align*}
where $\overline{\Delta} = \phi^{-1}\Delta_{\HH^n}\phi : L^2(\RR^n) \longrightarrow L^2(\RR^n)$. Writing $(t - \overline{\Delta})^{-1} u = u_1, (t - \Delta_{\RR^n})^{-1}u = u_2$, we get 
\begin{align*}
F(t) & = (-\overline{\Delta} u_1, u)_{\RR^n} - (-\Delta_{\RR^n} u_2, u)_{\RR^n}\\
& = (-\overline{\Delta} u_1, (t - \Delta_{\RR^n})u_2)_{\RR^n} - (-\Delta_{\RR^n} u_2, (t - \overline{\Delta})u_1)_{\RR^n}\\
& = t[(-\overline{\Delta} u_1, u_2)_{\RR^n} - (-\Delta_{\RR^n} u_2, u_1)_{\RR^n}].
\end{align*}
Writing $V(r) = V, K_1 = (\frac{n - 1}{2})(\frac{n - 3}{2}), K_2 = (\frac{n - 1}{2})^2$, we get from (\ref{oboseshe}) and (\ref{perechi}),
\begin{align*}
F(t)/t & = ((-\overline{\Delta}u_1 - (-\Delta_{\RR^n}))u_1, u_2)_{\RR^n}\\
& = ((-V\Delta_{S^{n - 1}} + K_1 V + K_2) u_1, u_2)_{\RR^n}\\
& = ((-V\Delta_{S^{n - 1}} + K_1 V + K_2)u_1, (t - \Delta_{\RR^n})^{-1}(t - \overline{\Delta})u_1)_{\RR^n}.
\end{align*}
Seeing that 
\[(t - \Delta_{\RR^n})^{-1}(t - \overline{\Delta}) = (t - \Delta_{\RR^n})^{-1}(t - \Delta_{\RR^n} + (-\overline{\Delta} - (-\Delta_{\RR^n})))\\
= I + (t - \Delta_{\RR^n})^{-1}(-\overline{\Delta} - (-\Delta_{\RR^n}))\]
we have 
\begin{align*}
F(t)/t & = ((-V\Delta_{S^{n - 1}} + K_1 V + K_2)u_1, (I + (t - \Delta_{\RR^n})^{-1}(-V\Delta_{S^{n - 1}} + K_1 V + K_2))u_1)_{\RR^n}\\
& = ((-V\Delta_{S^{n - 1}} + K_1 V + K_2)u_1, u_1)_{\RR^n} + ((-V\Delta_{S^{n - 1}} + K_1 V + K_2)u_1, (t - \Delta_{\RR^n})^{-1}\\
& \text{    }(-V\Delta_{S^{n - 1}} + K_1 V + K_2)u_1)_{\RR^n}\\
& = ((-V\Delta_{S^{n - 1}} + K_1 V + K_2)u_1, u_1)_{\RR^n} + ((t - \Delta_{\RR^n})w, w)_{\RR^n}\\
& > (V(-\Delta_{S^{n - 1}})u_1, u_1)_{\RR^n},
\end{align*}
where $w = (-V\Delta_{S^{n - 1}} + K_1 V + K_2)u_1$. If we now assume that $u_1$ is radial, then 
\[(V(-\Delta_{S^{n - 1}})u_1, u_1)_{\RR^n} = 0.\]
This means that $F(t)/t > 0$.\newline
Now, the reason that we can just choose $u_1$ radial in the above calculation is because we have
\beq\label{huhu}
W^{\text{sup}}_{\alpha,\HH^n} = \sup\{W_{\alpha,\HH^n}(u) : u \mbox{    is a radial function     } \in H^\alpha(\HH^n)\},
\eeq
and 
\beq\label{haha}
W^{\text{sup}}_{\alpha,\RR^n} = \sup\{W_{\alpha,\RR^n}(u) : u \mbox{    is a radial function     } \in H^\alpha(\RR^n)\}.
\eeq
(\ref{haha}) follows from (5.0.3) and (5.0.4) of ~\cite{CMMT1}.\newline
To show (\ref{huhu}), we need to verify that, replacing $u$ by the  radial decreasing rearrangement $u^*$ of $|u|$ lowers the kinetic energy term, that is, 
\[
\V (-\Delta_{\HH^n})^{\alpha/2}u^{*}\V^2_{L^2(\HH^n)} \leq \V (-\Delta_{\HH^n})^{\alpha/2} u\V^2_{L^2(\HH^n)}.
\]
This can be realized by the methods used in ~\cite{CM} as mentioned in the proof of Theorem \ref{WFMT}, in conjunction with the functional calculus used above. A proof more or less along such lines appears as Lemma 4.0.2 in ~\cite{CMMT1}, which we reproduce below. Taking this for granted, we have established that it is enough to compare the Weinstein functional values for radial functions in $H^\alpha (\RR^n)$ and $H^\alpha (\HH^n)$.\newline
Finally, we see that
\[
W_{\alpha, \RR^n}^{sup} = W_{\alpha, \HH^n}^{sup},
\]
and the corresponding fact that $W_{\alpha, \HH^n}^{sup}$ is not attained in $H^\alpha(\HH^n)$.
\end{proof}
The following lemma finishes the proof (for notational convenience,  in the lemma below, $-\Delta$ refers to $-\Delta_{\HH^n}$):
\begin{lemma} (~\cite{CMMT1}) Replacing $u \in H^\alpha(\HH^n)$ by the radial, decreasing rearrangement $u^*$ of $|u|$ lowers the term $\V (-\Delta)^{\frac{\alpha}{2}} u\V^2_{L^2(\HH^n)}$.
\end{lemma}
\begin{proof}
For $u \in H^\alpha(\HH^n)$, we have
\begin{align*}
\V (-\Delta)^{\frac{\alpha}{2}} u\V^2_{\HH^n} & = ((-\Delta)^\alpha u, u)_{\HH^n}\\
&  = \lim_{t \to 0} \frac{1}{t}((I - e^{-t(-\Delta)^\alpha})u, u)_{\HH^n}.
\end{align*}
To prove our lemma, it suffices to demonstrate that
\[
(e^{-t(-\Delta)^\alpha} u, u)_{\HH^n} \leq (e^{-t(-\Delta)^\alpha} u^*, u^*)_{\HH^n}.\]
Now,
\[
(e^{-t(-\Delta)^\alpha} u, u)_{\HH^n} = \int_{\HH^n}\int_{\HH^n} p_\alpha (t, \text{dist}(x, y)) u(x)u(y) dx dy,
\]
where $p_\alpha (t, \text{dist}(x, y))$ represents the integral kernel of the semigroup $e^{-t(-\Delta)^\alpha}$. We observe that
\[
e^{-t(-\Delta)^\alpha} = \int^\infty_0 f_{t, \alpha}(s) e^{s\Delta}ds, \text{   } t > 0,\]
with $f_{t, \alpha}(s) \geq 0$ (see ~\cite{Y}, pp. 260-261). So,
\[
e^{-t(-\Delta)^\alpha} u(x) = \int \bigg(\int^\infty_0 f_{t, \alpha}(s) p(t, \text{dist}(x, y))ds\bigg) u(y)dy,\]
which gives, 
\[
p_\alpha (t, \text{dist} (x, y)) = \int^\infty_0 f_{t, \alpha}(s)p (t, \text{dist} (x, y)) ds.\]
Hence, given $\alpha \in (0, 1)$, $t > 0$, and writing $r = \text{dist}(x, y)$, we have that $p_\alpha(t, r)$ is monotonically decreasing in $r$ (since we have from ~\cite{CM} that $p(t, r)$ is monotonically decreasing in $r$), and 
\[
p_\alpha (t, r) \geq 0.
\]
This gives, 
\[
(e^{-t(-\Delta)^\alpha} u, u)_{\HH^n} \leq (e^{-t(-\Delta)^\alpha} |u|, |u|)_{\HH^n}.\]
Now, we want to demonstrate that 
\[
(e^{-t(-\Delta)^\alpha} |u|, |u|)_{\HH^n} \leq (e^{-t(-\Delta)^\alpha} u^*, u^*)_{\HH^n}.\] 
But this follows from Theorem \ref{Drag}, by using $\Psi (f_1, f_2) = f_1f_2$ and $K_{12} = p_\alpha (r, t)$.
\end{proof}
\section{Appendix}
It might be of independent interest to observe whether a variant of the fractional Gagliardo-Nirenberg inequality holds on a non-compact rank 1 symmetric space, and particularly how the specific form of the heat kernel bears on this issue. To that end, we start by recalling that the heat kernel on a non-compact symmetric space of rank 1 satisfies (see ~\cite{HS})
\[
p(t, x, y) \leq (4\pi t)^{-n/2}e^{-d^2(x, y)/4t}\theta^{-1/2}(x, y)(1 + Ct),
\]
where $\theta : X \times X \longrightarrow (0, \infty)$ is defined by
\[
\theta(x, y) = \bigg| \text{det}\bigg(d(\text{Exp}_x)_{\text{Exp}^{-1}_x y}\bigg)\bigg|,
\]
where 
\[
d(\text{Exp}_x)_{\text{Exp}^{-1}_x y} : T_{\text{Exp}^{-1}_x y}(T_x X) \simeq T_x X \longrightarrow T_y X
\]
is an invertible linear map. Now, by Lemma 1 of ~\cite{HS}, we can clearly see that $\theta (x, y) \geq 1$, which gives
\[
p(t, x, y) \leq (4\pi t)^{-n/2}e^{-d^2(x, y)/4t}(1 + Ct).
\]
Now when $X = \RR^n$ or $\HH^n$, or even simply connected with nonpositive sectional curvature, it is known that $C = 0$. To tackle the general case, when {\bf $C \neq 0$}, we use the derivation in ~\cite{VSC} with certain modifications. For a very nice exposition of Varopoulos' proof and also including a proof of the Stein maximal ergodic theorem, see ~\cite{Bau}. As mentioned there, the noteworthy feature of Varopoulos' proof is to make use of the Stein maximal ergodic theorem bypassing the application of the more usual Marcinkiewicz interpolation techniques. We have
\begin{proposition}
On a non-compact symmetric space $X$ of rank 1, we have 
\[
\V u\V^{p + 1}_{L^{p + 1}} \lesssim \V (-\Delta)^{\alpha/2} u\V^{\theta(p + 1)}_{L^m}\V u\V^{(1 - \theta)(p + 1)}_{L^s},
\]
with 
\beq\label{rel}
\theta(\frac{1}{m} - \frac{\alpha}{n - 2}) + \frac{1 - \theta}{s} = \frac{1}{p + 1}.
\eeq
  
\end{proposition}
\begin{proof}
From what has gone in Proposition (\ref{2.4}), we are just content with proving the Hardy-Littlewood-Sobolev estimates
\beq \label{20}
\V u\V_{L^q} \lesssim \V (-\Delta)^{\alpha/2} u\V_{L^p},
\eeq
where $q = \frac{(n-2)p}{n - 2 - \alpha p}$.
 
We use the following functional calculus (see ~\cite{Str})
\[
I = (-\Delta)^{-\alpha/2} u = \frac{1}{\Gamma(\alpha/2)}\int_0^\infty t^{\alpha/2 - 1}e^{t\Delta}u dt.
\]
Now, we rewrite the above as 
\begin{align*}
I & = \frac{1}{\Gamma(\alpha/2)}\int_0^\delta t^{\alpha/2 - 1}e^{t\Delta}u dt + \frac{1}{\Gamma(\alpha/2)}\int_\delta^\infty t^{\alpha/2 - 1}e^{t\Delta}u dt\\
& = I_1 + I_2,
\end{align*}
where $\delta$ will be chosen later, but we will see that it can always be chosen such that $\delta > 1/C$. This gives, when $t > \delta, 1 + Ct < 2Ct$, so that 
\[
I_2 = \frac{1}{\Gamma(\alpha/2)}\int_\delta^\infty t^{\alpha/2 - 1}e^{t\Delta}u dt \lesssim \frac{C^{1/p}}{\Gamma(\alpha/2)}\int_\delta^\infty t^{\alpha/2 - 1 - (n - 2)/2p}dt \V u\V_{L^p},
\]
the last step following from the fact that 
\[
p(t, x, y) \lesssim Ct^{-(n - 2)/2} \Longrightarrow  e^{t\Delta}u \lesssim |\frac{C^{1/p}}{t^{(n - 2)/2}}| \V u\V_{L^p}.
\]
Now, 
\begin{align*}
|I_1| & \leq \frac{1}{\Gamma(\alpha/2)}\int_0^\delta t^{\alpha/2 - 1} dt |u^*(x)| \leq \frac{2}{\alpha\Gamma(\alpha/2)}\delta^{\alpha/2}|u^*(x)|,
\end{align*}
where $u^*(x) = \sup_{t \geq 0} |P_t u(x)|$, $P_t$ being a diffusion semigroup.
Also, 
\[
|I_2| \lesssim \frac{C^{1/p}}{\Gamma(\alpha/2)}\frac{1}{\frac{n - 2}{2p} - \frac{\alpha}{2}}\delta^{\alpha/2 - (n-2)/2p}\V u\V_{L^p}.
\]
Putting everything together, we have 
\[
|I| \lesssim \frac{2}{\alpha \Gamma(\alpha/2)}\delta^{\alpha/2}|u^*(x)| + \frac{C^{1/p}}{\Gamma(\alpha/2)}\frac{1}{\frac{n - 2}{2p} - \frac{\alpha}{2}}\delta^{\alpha/2 - (n-2)/2p}\V u\V_{L^p}.
\]
Now, we can solve for $\delta$ which gives equality in the power mean inequality, or, we can treat the right hand side in the above expression as a function of a single variable $\delta$ and optimise it as such. On calculation, we find 
\[
\delta^{-\frac{n - 2}{2p}} = \frac{|u^*(x)|}{C^{1/p}\V u\V_{L^p}}.\] It is also clear from the solution of $\delta$ that $\delta$ increases as $C$ increases, which implies that we can increase $C$ if we want and finally get a $\delta$ which satisfies $\delta C > 1$.\newline
Also, we have
\[
|I| \leq \frac{2nC^{\alpha/(n-2)}}{\alpha(n - 2 - p\alpha)\Gamma(\alpha/2)}\V u\V^{\alpha p/(n - 2)}_{L^p}|u^*(x)|^{1 - \alpha p/(n - 2)}.
\]
Now
\[
q = \frac{p(n - 2)}{n - 2 - p\alpha} \Longrightarrow 1 - \alpha p/(n - 2) = p/q,
\]
which gives, 
\[
|I|^q \lesssim \V u\V_{L^p}^{q - p}|u^*(x)|^p.
\]
Now, we apply the Stein maximal ergodic theorem, which states that
\[
\V u^*\V_{L^p} \leq \frac{p}{p-1}\V u\V_{L^p}, p > 1, u \in L^p.
\]
The application of this finally gives us
\[
\int_X |I|^q \lesssim \V u\V^{q - p}_{L^p}\V u\V^p_{L^p},
\]
which is actually the HLS estimate we want.

\end{proof}
  
\begin{remark}
To prove (\ref{20}), one could also interpolate between $\alpha = 0$ and $\alpha = 1$. The result, in any case, should be clear to the expert. But we included this proof as a showcase of the particular techniques it potrays. 
\end{remark}
\vspace{1 cm}
\subsection{Acknowledgements}
I wish to thank Jeremy Marzuola for going through a draft copy of this write-up and making several important suggestions, and also for teaching me about the fractional G-N inequality. I also thank my advisor Michael Taylor for his invaluable guidance throughout. \newline
\bibliographystyle{plain}


\end{document}